\documentclass[12pt, a4paper, reqno, oneside]{amsart}
\usepackage{amssymb}  
\usepackage{amsmath} 
\usepackage{amsthm}  
\usepackage{amsaddr}

\usepackage[cp1251]{inputenc}
\usepackage[english]{babel}
\usepackage{cite, verbatim}

\newtheorem{theorem}{Theorem}
\newtheorem{conj}{Conjecture}
\newtheorem{cor}{Corollary}
\newtheorem{lemma}{Lemma}
\numberwithin{lemma}{section}

\numberwithin{equation}{section}

\newcommand{\gexp}{\operatorname{exp}}

\newcommand{\Aut}{\operatorname{Aut}}

\begin{document}

\vspace{1cm}

\title[On the solvable radical of a group isospectral to a simple group]{On the nilpotency of the solvable radical of a finite group isospectral to a simple group}

\author{Nanying Yang}
\address{School of Science, Jiangnan University\\ Wuxi, 214122, P. R. China}

\author{Mariya A. Grechkoseeva}
\address{Sobolev Institute of Mathematics\\ Koptyuga 4, Novosibirsk 630090, Russia}

\author{Andrey V. Vasil'ev}
\address{Sobolev Institute of Mathematics\\ Koptyuga 4, Novosibirsk 630090, Russia\\
Novosibirsk State University\\ Pirogova 1, Novosibirsk 630090, Russia}

\email{yangny@jiangnan.edu.cn, grechkoseeva@gmail.com, vasand@math.nsc.ru}

\thanks{The first author was supported by NNSF grant of China (No. 11301227), the second author was supported by RFBR according to the research project no. 18-31-20011, and the third author was supported by Foreign Experts program in Jiangsu Province (No. JSB2018014).}

\begin{abstract}
We refer to the set of the orders of elements of a finite group as its spectrum and say that groups are isospectral if their spectra coincide.
We prove that with the only specific exception the solvable radical of a nonsolvable finite group isospectral to a finite simple group is nilpotent.

 \end{abstract}

\maketitle

\section{Introduction}

In 1957 G.~Higman \cite{57Hig} investigated finite groups in which every element has prime power order (later they were called the $CP$-{\em{groups}}). He gave a description of solvable $CP$-groups by showing that any such group is a $p$-group, or Frobenius, or $2$-Frobenius, and its order has at most two distinct prime divisors. Concerning a nonsolvable group $G$ with the same property, he proved that $G$ has the following structure:
\begin{equation}\label{eq:main}
1\leqslant K< H\leqslant G,
\end{equation}
where the solvable radical $K$ of $G$ is a $p$-group for some prime $p$, $H/K$ is a unique minimal normal subgroup of $G/K$ and is isomorphic to some nonabelian simple group $S$, and $G/H$ is cyclic or generalized quaternion. Later, Suzuki, in his seminal paper \cite{62Suz}, where the new class of finite simple groups (now known as the Suzuki groups) was presented, found all nonabelian simple $CP$-groups. The exhaustive description of $CP$-groups was completed by Brandl in 1981 \cite{81Bra}. It turns out that there are only eight possibilities for nonabelian composition factor $S$ in~\eqref{eq:main}: $L_2(q)$, $q=4,7,8,9,17$, $L_3(4)$, $Sz(q)$, $q=8,32$; the solvable radical $K$ must be a $2$-group (possibly trivial), and there is only one $CP$-group with nontrivial factor $G/H$, namely, $M_{10}$, an automorphic extension of $A_6$.

In the middle of 1970s Gruenberg and Kegel invented the notion of the prime graph of a finite group (nowadays it is also called the Gruenberg--Kegel graph) and noticed that for groups with disconnected prime graph the very similar results to Higman's ones can be proved. Recall that the {\em prime graph} $GK(G)$ of a finite group $G$ is a labelled graph whose vertex set is $\pi(G)$, the set of all prime divisors of $|G|$, and in which two different vertices labelled by $r$ and $s$ are adjacent if and only if $G$ contains an element of order~$rs$. So, according to this definition, $G$ is a $CP$-group if and only if $GK(G)$ is a coclique (all vertices are pairwise nonadjacent). Gruenberg and Kegel proved that a solvable finite group $G$ with disconnected prime graph is Frobenius or $2$-Frobenius and the number of connected components equals~$2$ (cf. Higman's result), while a nonsolvable such group has again a normal series~(\ref{eq:main}), where the solvable radical $K$ is a nilpotent $\pi_1$-group (here $\pi_1$ is the vertex set of the connected component of $GK(G)$ containing $2$), $H/K$ is a unique minimal normal subgroup of $G/K$ and is isomorphic to some nonabelian simple group $S$, and $G/H$ is a $\pi_1$-group. The above results were published for the first time by Gruenberg's student Williams in~\cite{81Wil}. There he also started the classification of finite simple groups with disconnected prime graph that was completed by Kondrat'ev in 1989~\cite{89Kon.t}  (see \cite[Tables 1a--1c]{02Maz.t} for a revised version). Though many nonabelian simple groups, for example, all sporadic ones, have disconnected prime graph, there is a bulk of classical and alternating simple groups which do not enjoy this property. Nevertheless, as we will see below, if an arbitrary finite group has the set of orders of elements as a nonabelian simple group, then its structure can be described as Higman's and Gruenberg--Kegel's theorems do.

For convenience,  we refer to the set $\omega(G)$ of the orders of elements of a finite group $G$ as its {\em spectrum} and say that groups are {\em isospectral} if their spectra coincide. It turns out that there are only three finite nonabelian simple groups $L$, namely, $L_3(3)$, $U_3(3)$, and $S_4(3)$, that have the spectrum as some solvable group \cite[Corollary~1]{10Zav.t} (again the latter must be Frobenius or $2$-Frobenius). It is also known that a nonsolvable group $G$ isospectral to an arbitrary nonabelian simple group $L$ has a normal series~(\ref{eq:main}) with the only nonabelian composition factor $H/K$ (see, e.g, \cite[Lemma~2.2]{15VasGr1}). Here we are interested in the nilpotency of the solvable radical $K$ of $G$.

\begin{theorem} \label{t:main} Let $L$ be a finite nonabelian simple group distinct from the alternating group~$A_{10}$. If $G$ is a finite nonsolvable group with $\omega(G)=\omega(L)$, then the
solvable radical $K$ of $G$ is nilpotent.
\end{theorem}

Observe that, as shown in \cite[Proposition~2]{98Maz.t} (see \cite{10Sta.t} for details), there is a nonsolvable group having a non-nilpotent solvable radical and isospectral to the alternating group $A_{10}$.

Theorem 1 together with the aforementioned results gives the following

\begin{cor}\label{cor:main} Let $L$ be a finite nonabelian simple group distinct from $L_3(3)$, $U_3(3)$, and $S_4(3)$. If $G$ is a finite group with $\omega(G)=\omega(L)$, then there is a nonabelian simple group $S$ such that $$S\leqslant G/K\leqslant \Aut S,$$ where $K$ being
the largest normal solvable subgroup of $G$ is nilpotent provided $L\neq A_{10}$.
\end{cor}


As in the case of $CP$-groups, the thorough analysis of groups isospectral to simple ones allows to say more. Though, there are quite a few examples of finite groups with nontrivial solvable radical which are
isospectral to nonabelian simple groups (see \cite[Table~1]{18Lyt}), in general the situation is much better. In order to describe it, we refer to a nonabelian simple group $L$ as {\em recognizable
(by spectrum)} if every finite group $G$ isospectral to $L$ is isomorphic to $L$, and as {\em almost recognizable (by spectrum)} if every such a group $G$ is an almost simple group with socle isomorphic
to $L$. It is known that all sporadic and alternating groups, except for $J_2$, $A_6$, and $A_{10}$, are recognizable (see \cite{98MazShi, 13Gor.t}), and all exceptional groups excluding
$^3D_4(2)$ are almost recognizable (see \cite{14VasSt.t,16Zve.t}). In 2007 Mazurov conjectured that there is a positive integer $n_0$ such that all simple classical groups of dimension at
least $n_0$ are almost recognizable as well. Mazurov's conjecture was proved in \cite[Theorem 1.1]{15VasGr1} with $n_0=62$. Later it was shown \cite[Theorem 1.2]{17Sta} that we can take
$n_0=38$. It is clear that this bound is far from being final, and we conjectured that the following holds \cite[Conjecture 1]{15VasGr1}.

\begin{conj} \label{c:1} Suppose that $L$ is one of the following nonabelian simple groups:
\begin{enumerate}
\item $L_n(q)$, where $n\geqslant5$;
\item $U_n(q)$, where $n\geqslant5$ and $(n,q)\neq(5,2)$;
\item $S_{2n}(q)$, where $n\geqslant3$, $n\neq4$ and $(n,q)\neq(3,2)$;
\item $O_{2n+1}(q)$, where $q$ is odd, $n\geqslant3$, $n\neq4$ and $(n,q)\neq(3,3)$;
\item $O_{2n}^\varepsilon(q)$, where $n\geqslant4$ and $(n,q,\varepsilon)\neq (4,2,+),(4,3,+)$.
\end{enumerate}
Then every finite group isospectral to $L$ is an almost simple group with socle isomorphic to $L$.
\end{conj}

In order to prove the almost recognizability of a simple group $L$ one should prove the triviality of the solvable radical $K$ of a group isospectral to $L$. It does not sound surprising that establishing the nilpotency of $K$ is a necessary step toward that task (see, e.g., \cite{19GrVasZv}). Thus our main result, besides everything, provides a helpful tool for the verification of the conjecture.

\section{Preliminaries}

As usual, $[m_1,m_2,\dots,m_k]$ and $(m_1,m_2,\dots,m_k)$ denote respectively the least common multiple and greatest common divisor of integers $m_1,m_2,\dots,m_s$.
For a positive integer~$m$, we write $\pi(m)$ for the set of prime divisors of $m$. Given a prime $r$, we write $(m)_{r}$ for the $r$-part of $m$, that is, the highest power of $r$ dividing $m$,
and $(m)_{r'}$ for the $r'$-part of $m$, that is, the ratio $m/(m)_{r}$. If $\varepsilon\in\{+,-\}$, then in arithmetic expressions, we abbreviate $\varepsilon 1$ to $\varepsilon$.
The next lemma  is well known (see, for example, \cite[Ch. IX, Lemma 8.1]{82HupBl2}).

\begin{lemma}\label{l:r-part}
Let $a$ and $m$ be positive integers and let $a>1$. Suppose that $r$ is a prime and  $a\equiv \varepsilon\pmod r$, where  $\varepsilon\in\{+1,-1\}$.
\begin{enumerate}
\item If $r$ is odd, then $(a^m-\varepsilon^m)_{r}=(m)_r(a-\varepsilon)_{r}$.
\item If $a\equiv \varepsilon\pmod 4$, then $(a^m-\varepsilon^m)_{2}=(m)_{2}(a-\varepsilon)_{2}$.
\end{enumerate}
\end{lemma}

Let $a$ be an integer. If $r$ is an odd prime and $(a,r)=1$, then $e(r,a)$ denotes the multiplicative order of
$a$ modulo $r$. Define $e(2,a)$ to be 1 if  $4$ divides $a-1$ and to be $2$ if $4$ divides $a+1$.
A primitive prime divisor of $a^m-1$, where $|a|>1$ and $m\geqslant 1$, is a prime $r$ such that $e(r,a)=m$.
The set of primitive prime divisors of $a^m-1$ is denoted by $R_m(a)$,
and we write $r_m(a)$ for an element of $R_m(a)$ (provided that it is not empty). The following well-known lemma was proved in~\cite{86Bang} and independently in~\cite{Zs}.

\begin{lemma}[Bang--Zsigmondy]\label{l:bz}
Let $a$ and $m$ be integers, $|a|>1$ and $m\geqslant 1$. Then the set $R_m(a)$ is not empty, except when $$(a,m)\in\{(2,1),(2,6),(-2,2),(-2,3),(3,1),(-3,2)\}.$$
\end{lemma}

\begin{lemma}\label{l:pras}
Let $k,m,a$ be positive integers numbers, $a>1$. Then $R_{mk}(a)\subseteq R_m(a^k)$. If, in addition, $(m,k)=1$, then $R_{m}(a)\subseteq R_m(a^k)$.
\end{lemma}

\begin{proof} It easily follows from the definition of $R_m(a)$ (see, e.g., \cite[Lemma~6]{12GrLyt.t}).
\end{proof}

The largest primitive divisor of $a^m-1$, where $|a|>1$, $m\geqslant 1$, is the number $k_m(a)=\prod_{r\in R_m(a)}|a^m-1|_{r}$ if $m\neq 2$, and
the number $k_2(a)=\prod_{r\in R_2(a)}|a+1|_{r}$ if $m=2$. The largest primitive divisors can be written in terms of cyclotomic polynomials~$\Phi_m(x)$.

\begin{lemma}\label{l:k_n}
Let $a$ and $m$ be integers, $|a|>1$ and $m\geqslant 3$. Suppose that $r$ is the largest prime divisor of $m$ and $l=(m)_{r'}$. Then
$$k_m(a)=\frac{|\Phi_m(a)|}{(r,\Phi_l(a))}.$$
Furthermore, $(r,\Phi_l(a))=1$ whenever $l$ does not divide $r-1$.
\end{lemma}

\begin{proof}
This follows from \cite[Proposition 2]{97Roi} (see, for example, \cite[Lemma 2.2]{15VasGr.t}).
\end{proof}

Recall that $\omega(G)$ is the set of the orders of elements of $G$. We write $\mu(G)$ for the set of maximal under divisibility elements of $\omega(G)$. The least common multiple of the elements of $\omega(G)$ is equal to
the exponent of $G$ and denoted by $\gexp(G)$. Given a prime $r$, $\omega_{r}(G)$ and $\gexp_r(G)$ are respectively the spectrum and the exponent of a Sylow $r$-subgroup of $G$. Similarly, $\omega_{r'}(G)$ and $\gexp_{r'}(G)$ are respectively the set of the orders of elements of $G$ that are coprime to $r$ and the least common multiple of these orders.

A coclique of a graph is a set of pairwise nonadjacent vertices. Define $t(G)$ to be the largest size of a coclique of the prime graph $GK(G)$ of a finite group $G$. Similarly, given $r\in\pi(G)$, we write
$t(r,G)$ for the largest size of a coclique of $G$ containing $r$. It was proved in \cite{05Vas.t} that a finite group $G$ with $t(G)\geqslant 3$ and $t(2,G)\geqslant 2$ has exactly one nonabelian composition factor.
Below we provides the refined version of this assertion from \cite{09VasGor.t}.

\begin{lemma}[{\cite{05Vas.t, 09VasGor.t}}]\label{l:str}
Let $L$ be a finite nonabelian simple group such that $t(L)\geqslant3$ and $t(2,L)\geqslant2$, and suppose that a finite group $G$
satisfies $\omega(G)=\omega(L)$. Then the following holds.
\begin{enumerate}
\item There is a nonabelian simple group $S$ such that $$S\leqslant \overline G=G/K\leqslant \Aut S,$$ with $K$ being
the largest normal solvable subgroup of $G$.

\item If $\rho$ is a coclique of $GK(G)$ of size at least  $3$, then at most one prime of $\rho$
divides $|K|\cdot|\overline{G}/S|$. In particular, $t(S)\geqslant t(G)-1$.

\item If $r\in\pi(G)$ is not adjacent to $2$ in $GK(G)$, then $r$ does not divide
$|K|\cdot|\overline{G}/S|$. In particular,  $t(2,S)\geqslant t(2,G)$.
\end{enumerate}
\end{lemma}

The next lemma summarizes what we know about almost recognizable simple groups (see \cite{15VasGr1,15VasGr.t,17Sta}).

\begin{lemma}\label{l:isospectral} Let $L$ be one of the following nonabelian simple groups:
\begin{enumerate}
\item \label{t:sporadic} a sporadic group other than $J_2$;
\item \label{t:alternating} an alternating group $A_n$, where $n\neq6,10$;
\item \label{t:exceptional} an exceptional group of Lie type other than ${}^3D_4(2)$;
\item\label{t:linear} $L_n(q)$, where $n\geqslant27$ or $q$ is even;
\item\label{t:unitary} $U_n(q)$, where $n\geqslant27$, or $q$ is even and $(n,q)\neq (4,2),(5,2)$;
\item\label{t:BnCn} $S_{2n}(q),O_{2n+1}(q)$, where either $q$ is odd and $n\geqslant16$, or $q$ is even and $n\neq2,4$ and $(n,q)\neq(3,2)$;
\item\label{t:Dn} $O^+_{2n}(q)$, where either $q$ is odd and $n\geqslant19$, or $q$ is even and $(n,q)\neq(4,2)$;
\item\label{t:2Dn} $O^-_{2n}(q)$, where either $q$ is odd and $n\geqslant18$, or $q$ is even.
\end{enumerate}
Then every finite group isospectral to $L$ is isomorphic to some group $G$ with $L\leqslant G\leqslant \Aut L$. In particular,
there are only finitely many pairwise nonisomorphic finite groups isospectral to $L$.
\end{lemma}

Now we list the spectra of some groups of low Lie rank and give some lower bounds on the exponents of exceptional groups of Lie type.
Throughout the paper we repeatedly use, mostly without explicit references,  the description of the spectra of simple classical groups from \cite{10But.t} (with corrections from \cite[Lemma 2.3]{16Gr.t}) and \cite{08But.t},
as well as the adjacency criterion for the prime graphs of simple groups of Lie type from \cite{05VasVd.t} (with corrections from \cite{11VasVd.t}).
Also we use the abbreviations $L_n^\tau(u)$ and $E_6^\tau(u)$, where $\tau\in\{+,-\}$, that are defined as follows:  $L_n^+(u)=L_n(u)$, $L_n^-(u)=U_n(u)$, $E_6^+(u)=E_6(u)$ and $E_6^-(u)={}^2E_6(u)$.

\begin{lemma}[{\cite{08But.t}}]\label{l:spec_l4}
Let $q$ be a power of an odd prime $p$ and let $L=L_4^\tau(q)$. The set $\omega(L)$ consists of the divisors of the following numbers:
\begin{enumerate}
  \item $(q^2+1)(q+\tau)/(4,q-\tau)$, $(q^3-\tau)/(4,q-\tau)$, $q^2-1$, $p(q^2-1)/(4,q-\tau)$, $p(q-\tau)$;
  \item $9$ if $p=3$.
\end{enumerate}
In particular, $\gexp_{p'}(L)=(q^4-1)(q^2+\tau q+1)/2$.
\end{lemma}

\begin{lemma}[{\cite{08But.t}}]\label{l:spec_l6}
Let $q$ be a power of an odd prime $p$ and let $L=L_6^\tau(q)$. The set $\omega(L)$ consists of the divisors of the following numbers:
\begin{enumerate}
  \item $(q^3+\tau)(q^2+\tau q+1)/(6,q-\tau)$, $(q^5-\tau)/(6,q-\tau)$,
  $q^4-1$, $p(q^4-1)/(6,q-\tau)$, $p(q^3-\tau)$, $p(q^2-1)$;
  \item $p^2$ if $p=3,5$.
  \end{enumerate}
In particular, $\gexp_{p'}(L)=(q^6-1)(q^5-\tau)(q^2+1)/(q-\tau)$.
\end{lemma}

\begin{lemma}[{\rm \cite{62Suz}}]\label{l:spec_2b2}
Let $u=2^{2k+1}\geqslant 8$. Then $\omega({}^2B_2(u))$ consists of the divisors of the numbers $4$, $u-1$, $u-\sqrt{2u}+1$, and $u+\sqrt{2u}+1$.
In particular, $\gexp({}^2B_2(u))=4(u^2+1)(u-1)$.
\end{lemma}

\begin{lemma}\label{l:spec_g2}
Let $u$ be a power of a prime $v$. Then $\omega(G_2(u))$ consists of the divisors of the numbers $u^2\pm u+1$,
$u^2-1$, and $v(u\pm 1)$ together with the divisors of
\begin{enumerate}
 \item $8$, $12$ if $v=2$;
 \item $v^2$ if $v=3,5$.
\end{enumerate}
In particular, $\gexp_{v'}(S)=(u^6-1)/(3,u^2-1)$. Furthermore, if a Sylow $r$-subgroup of $G_2(u)$ is cyclic, then $r$ divides $u^2+u+1$ or $u^2-u+1$.
\end{lemma}

\begin{proof}
 See \cite{84Der} and \cite{87Asch}.
\end{proof}

\begin{lemma}\label{l:exp_e} If $S$ and $f(u)$ are as follows, then $\gexp(S)>f(u)$.
\renewcommand{\arraystretch}{1.5}
$$\begin{array}{|c|cccccc|}
\hline
S& E_8(u)& E_7(u)& E_6^\pm(u)& F_4(u)& ^2F_4(u)& {}^2G_2(u)\\
\hline
f(u)& 2u^{80}& 3u^{48}& u^{26}& 3u^{16}& 5u^{10}& 2u^4\\
\hline
\end{array}$$
\end{lemma}

\begin{proof}
If $S\neq F_4(u)$, $^2G_2(u)$, the assertion is proved in \cite[Lemma 3.6]{19GrVasZv}. It follows from \cite{66Wa} that $\gexp({}^2G_2(u))=9(u^3+1)(u-1)/4$,
and so $\gexp({}^2G_2(u))>9(u^4-u^3)/4>2u^4$.

Let $S=F_4(u)$ and let $u$ be a power of a prime $v$. Using \cite{84Der} and Lemma \ref{l:r-part}, it is not hard to see that $\gexp_v(F_4(u))\geqslant 13$ and
$\gexp_{v'}(F_4(u))=(u^{12}-1)(u^4+1)/(2,u-1)^2$, so we have the desired bound.
\end{proof}

\begin{lemma}\label{l:v}
Let $S$ be a finite simple group of Lie type. If $r,s,t\in\pi(S)$ and $rt,st\in\omega(S)$, but $rs\not\in\omega(S)$, then a Sylow $t$-subgroup of $S$ is not cyclic.
\end{lemma}

\begin{proof} For classical groups with $n=\operatorname{prk}(S)\geqslant4$ it easily follows as $\varphi(t,S)\leqslant n/2$ (in the sense of~\cite{15Vas}). For exceptional groups of Lie type and
classical groups with $\operatorname{prk}(S)\leq3$, this can be checked directly with the help of \cite{05VasVd.t,11VasVd.t} and known information on the structure of their maximal tori.
\end{proof}

The next five lemmas are tools for calculating the orders of elements in group extensions. Most of them are corollaries of well-known results (such as the Hall--Higman theorem).

\begin{lemma}\label{l:nonc} Suppose that $G$ is a finite group,  $K$ is a normal subgroup of $G$
and $w\in\pi(K)$. If $G/K$ has a noncyclic Sylow $t$-subgroup
for some odd prime $t\neq w$, then $tw\in\omega(G)$.
\end{lemma}

\begin{proof}
Let $W$ be a Sylow $w$-subgroup of $K$ and $T$ be a Sylow $t$-subgroups of $N_G(W)$. By the Frattini argument, $G=KN_G(W)$, and so $T$
is also noncyclic. By the classification of Frobenius complements, $T$ cannot act on $W$  fixed-point-freely, therefore, $tw\in\omega(G)$.
\end{proof}

\begin{lemma} \label{l:frob} Suppose that $G$ is a finite group,  $K$ is a normal $w$-subgroup of $G$
for some prime $w$ and $G/K$ is a Frobenius group with kernel $F$ and cyclic complement $C$. If
$(|F|, w) = 1$ and $F$ is not contained in $KC_G(K)/K$, then $w|C| \in\omega(G)$.
\end{lemma}

\begin{proof}
See \cite[Lemma 1]{97Maz.t}.
\end{proof}

\begin{lemma}\label{l:hh} Let $v$ and $r$ be distinct primes and let $G$ be a semidirect product of a finite $v$-group $U$ and a cyclic group $\langle g\rangle$ of order $r$.
Suppose that $[U,g]\neq 1$ and $G$ acts faithfully on a vector space $W$ of positive characteristic  $w\neq v$. Then either the natural semidirect product $W\rtimes G$ has an element of order $rw$,
or the following holds:
\begin{enumerate}
 \item $C_U(g)\neq 1$;
 \item $U$ is nonabelian;
 \item $v=2$ and $r$ is a Fermat prime.
\end{enumerate}
\end{lemma}

\begin{proof}
See \cite[Lemma 3.6]{15Vas}.
\end{proof}

\begin{lemma}\label{l:hh1} Let $G$ be a finite group, let $N$ be a normal subgroup of $G$ and let $G/N$ be a simple classical group over a field of characteristic $v$.
Suppose that $G$ acts on a vector space $W$ of positive characteristic $w$,  $r$ is an odd prime dividing the order of some proper parabolic subgroup of $G/N$, the primes
$v$, $w$, and $r$ are distinct, and $v,r\not\in\pi(N)$. Then the natural semidirect product $W\rtimes G$ has an element of order $rw$.
\end{lemma}

\begin{proof}
Let $S=G/N$. We may assume that $C_G(W)\subseteq N$, since otherwise $C_G(W)$ has $S$ as a section and the lemma follows.
Let $P$ be the proper parabolic subgroup of $S$ contaning an element $g$ of order $r$ and let $U$ be the unipotent radical of $P$. By \cite[13.2]{83GorLy}, it follows that
$[U,g]\neq 1$. Since both $v$ and $r$ are coprime to $N$, there is a subgroup of $G$ isomorphic to $U\rtimes\langle g\rangle$ and this subgroup acts
on $W$ faithfully.  Applying Lemma \ref{l:hh}, we see that either $rw\in\omega(W\rtimes G)$, or $v=2$, $r$ is a Fermat prime, $U$ is nonabelian
and $C_U(g)\neq 1$.

Suppose that the latter case holds. If $S\neq U_n(u)$, then the conditions $v=2$ and $2r\in\omega(S)$ imply that $r$ divides the order of a maximal parabolic subgroup of $S$
with abelian unipotent radical (for example, the order of the group $P_1$ in notation of \cite{90KlLie}), and we can proceed as above with this parabolic subgroup instead of $P$.
Let $S=U_n(u)$. Writing $k=e(r,-u)$, we have that $k\leqslant n-2$ and $k$ divides $t-1$. Since $r$ is a Fermat prime, it follows that $k=1$ or $k$ is even, and
so $r$ divides $u^2-1$ or $u^k-1$. If $n\geqslant 4$, then $S$ includes a Frobenius group whose kernel is a $v$-group and complement has order $r$, and we again can apply Lemma~\ref{l:hh}. Let $n=3$. Then $r$ divides $u+1$. Since $|S|_r=(u+1)_r^2$ and  $\gexp_r(S)=(u+1)_r$, Sylow $r$-subgroups of $S$ are not cyclic,  and applying Lemma \ref{l:nonc} completes the proof.
\end{proof}

\begin{lemma}\label{l:l_2v}
Let $G=L_2(v)$, where $v>3$ is a prime. Suppose that $G$ acts on a vector space $W$ of positive characteristic  $w$ and $w\not\in\pi(G)$.
If $r$ and $s$ are two distinct odd primes from $\pi(G)$, then $\{r,s,w\}$ is not a coclique in $GK(W\rtimes G)$.
\end{lemma}

\begin{proof}
Clearly, we may assume that $G$ acts on $W$ faithfully. If one of $r$ and $s$, say $r$, divides $(v-1)/2$, then applying Lemma  \ref{l:hh} to
a Borel subgroup of $G$ yields $wr\in GK(W\rtimes G)$. If both $r$ and $s$ divide $(v+1)/2$, then $rs\in \omega(G)$. Thus we are left with the case where one of $r$ and $s$ is equal to $v$, while
another divides $(v+1)/2$. We prove that either an element of order $v$ or an element of order $(v+1)/2$ has a fixed point in $W$. We may assume that $W$ is an irreducible $G$-module. The ordinary
character table of $L_2(v)$ is well known. We use the result and notation due to Jordan \cite{07Jor}.

Define $\epsilon$ to be $+1$ or $-1$ depending on whether $v-1$ is even or odd. Also let $\zeta$ be some fixed not-square in the field of order $v$. We denote by $\mu$ and $\nu$,
respectively, the conjugacy classes containing the images in $L_2(v)$ of the matrices $$\begin{pmatrix} 1&1\\0&1\end{pmatrix} \text{ and }\begin{pmatrix} 1&\zeta\\0&1\end{pmatrix}.$$
We denote by $S^b$ with $1\leqslant b\leqslant (v-\epsilon)/4$ the conjugacy class containing the image of the $b$-power of some fixed element of order $(v+1)/2$.
Then the values of nontrivial irreducible characters of $G$ on elements of orders dividing $v$ and $(v+1)/2$ are as follows:

\renewcommand{\arraystretch}{1.5}
$$
\begin{array}{|l|cccc|}
\hline
& 1& \mu&\nu&  S^b\\
\hline
\chi_0&v&0&0&-1\\
\chi_{\pm}&\frac{v+\epsilon}{2}&\frac{\epsilon\pm\sqrt{\epsilon v}}{2}&\frac{\epsilon\mp\sqrt{\epsilon v}}{2}&-\frac{(1-\epsilon)(-1)^b}{2}\\
\chi_u&v+1&1&1&0\\
\chi_t&v-1&-1&-1&-t^b-t^{-b}\\
\hline
\end{array}
$$
where $u$ and $t$ are the roots (except $\pm 1$) of the respective equations $u^{(v-1)/2}=1$ and $t^{(v+1)/2}=1$.

If $g\in G$ and $\chi$ is the character of the representation on $W$, then $C_W(g)\neq 0$ if and only if the sum \begin{equation}\label{e:sum}\sum_{h\in\langle g\rangle}\chi(h)\end{equation} is
positive. If $g\in\mu$, then the sum  is equal to $$\chi(1)+\left(\frac{v-1}{2}\right)\cdot(\chi(\mu)+\chi(\nu)),$$ and so it is clearly positive unless $\chi=\chi_{\pm}$ and $\epsilon=-1$, or
$\chi=\chi_t$.

Let $g\in S^1$. If $\epsilon=-1$, then \eqref{e:sum} is equal to $$\chi(1)+\chi(S^{(v+1)/4})+2\sum_{1\leqslant b\leqslant \frac{v-3}{4}}\chi(S^b).$$
Taking $\chi=\chi_{\pm}$, we have $$\frac{v-1}{2}-(-1)^{(v+1)/4}-2\sum_{1\leqslant b\leqslant \frac{v-3}{4}}(-1)^b=\frac{v-1}{2}-\sum_{1\leqslant b\leqslant \frac{v+1}{2}}(-1)^b=\frac{v+1}{2}.$$
Similarly, if $\chi=\chi_t$, then we have $$v-1-2t^{(v+1)/4}-2\sum_{1\leqslant b\leqslant \frac{v-3}{4}}(t^b+t^{-b})=v-1-2\sum_{1\leqslant b\leqslant
\frac{v-1}{2}}t^b=v+1.$$ If $\epsilon=+1$ and $\chi=\chi_t$, then \eqref{e:sum} is equal to
$$\chi_t(1)+2\sum_{1\leqslant b\leqslant \frac{v-1}{4}}\chi_t(S^b)=v-1-2\sum_{1\leqslant b\leqslant
\frac{v-1}{2}}t^b=v+1.$$
The proof is complete.
\end{proof}

We conclude the section with the lemma from \cite[Lemma~2.9]{19GrVasZv}. We give it with the proof because we will use variations of this proof further.

\begin{lemma}\label{l:nilp}
Let $G$ be a finite group and let $S\leqslant G/K\leqslant \Aut S$, where $K$ is a normal solvable subgroup of $G$ and $S$ is a nonabelian simple group.
Suppose that for every $r\in\pi(K)$, there is $a\in\omega(S)$ such that $\pi(a)\cap\pi(K)=\varnothing$
and $ar\not\in\omega(G)$. Then $K$ is nilpotent.
\end{lemma}

\begin{proof}
Otherwise, the Fitting subgroup $F$ of $K$ is a proper subgroup of $K$.
Define $\widetilde G=G/F$ and $\widetilde K=K/F$. Let $\widetilde T$ be a minimal normal subgroup of
$\widetilde G$ contained in $\widetilde K$ and let $T$ be its preimage in $G$. It is clear that $\widetilde T$ is an elementary abelian $t$-group
for some prime $t$. Given $r\in \pi(F)\setminus \{t\}$, denote the Sylow $r$-subgroup of $F$ by $R$,
its centralizer in $G$ by $C_r$ and the image of $C_r$ in $\widetilde G$ by $\widetilde C_r$. Since $\widetilde C_r$ is normal in $\widetilde G$,
it follows that either $\widetilde T\leqslant \widetilde C_r $ or $\widetilde C_r \cap \widetilde T = 1$.
If $\widetilde T\leqslant \widetilde C_r$ for all $r\in\pi(F)\setminus \{t\}$, then $T$ is a normal nilpotent subgroup of $K$, which contradict the choice of $\widetilde T$.
Thus there is $r\in\pi(F)\setminus\{t\}$ such that $\widetilde C_r \cap \widetilde T = 1$.

If $C_{\widetilde G}(\widetilde T)$ is not contained in $\widetilde K$, then it has a section isomorphic to $S$. In this case $ta\in\omega(G)$ for every $a\in\omega_{t'}(S)$, contrary
to the hypothesis. Thus  $C_{\widetilde G}(\widetilde T)\leqslant \widetilde K$.

Choose $a\in\omega_{r'}(S)$ such that $\pi(a)\cap\pi(K)=\varnothing$ and $ra\not\in\omega(G)$, and let $x\in \widetilde G$ be an element of order $a$. Then
$x\not\in C_{\widetilde G}(\widetilde T)$, therefore, $[\widetilde T,x]\neq 1$ and so $[\widetilde T,x]\rtimes\langle x\rangle$ is a Frobenius group with complement $\langle x\rangle$.
Since $\widetilde C_r \cap \widetilde T = 1$, we can apply Lemma \ref{l:frob} and conclude that $ra\in\omega(G)$, contrary to the choice of $a$.
\end{proof}

\section{Reduction}

In this section, we apply the known facts concerning Theorem~\ref{t:main} to reduce the general situation to a special case. Let $L$ be a finite nonabelian simple group and let
$G$ be a nonsolvable finite group with $\omega(G)=\omega(L)$. Since the spectrum of a group determines its prime graph, it follows that $GK(G)=GK(L)$. In particular, if $GK(L)$ is disconnected, then
so is $GK(G)$. In this case, $G$ satisfies the hypothesis of the Gruenberg--Kegel theorem, hence the solvable radical $K$ of $G$ must be nilpotent (see \cite[Theorem~A and
Lemma~3]{81Wil}). Thus, proving Theorem~\ref{t:main} we may assume that $GK(L)$ is connected.

If $L$ is sporadic, alternating, or exceptional group of Lie type, then Lemma~\ref{l:isospectral} says that $L$ is either almost recognizable, or one of the groups $J_2$, $A_6$, $A_{10}$, and
${}^3D_4(2)$. In the former case the solvable radical $K$ is trivial, while in the latter case, if we exclude $A_{10}$, then $K$ is nilpotent because $GK(L)$ is disconnected \cite[Tables 1a-1c]{02Maz.t}. Thus, we
may suppose that $L$ is a classical group.

Let $p$ and $q$ be the characteristic and order of the base field of $L$, respectively.
If $L$ is one of the following groups: $L_n^\tau(q)$, where $n\leq3$, $S_{2n}(q), O_{2n+1}(q)$, where $n=2,4$, $U_4(2)$, $U_5(2)$, $S_6(2)$, and $O_{8}^+(2)$, then $GK(L)$ is
disconnected \cite{02Maz.t}. Together with Lemma~\ref{l:isospectral}, this shows that we may assume that $q$ is odd, and $n\geq4$ for $L=L_n^\tau(q)$, $n\geq3$ for $L\in\{S_{2n}(q), O_{2n+1}(q)\}$,
and $n\geqslant4$ for $L=O_{2n}^\tau(q)$. Furthermore, applying information on the sizes of maximal cocliques and $2$-cocliques from \cite{05VasVd.t,11VasVd.t}, we obtain that $t(L)\geqslant3$ and
$t(2,L)\geqslant2$, so the conclusion of Lemma~\ref{l:str} holds for $G$. In particular, $G$ has a normal series
\begin{equation} \label{e:ns}
1\leqslant K<H\leqslant G,
\end{equation}
where $K$ is the solvable radical of $G$, $H/K$ is isomorphic to a nonabelian simple group $S$, and
$G/K$ is isomorphic to some subgroup of $\Aut(S)$.

The group $S$ is neither an alternating group by \cite[Theorem~1]{09VasGrMaz.t} and \cite[Theorem~1]{11VasGrSt.t}, nor a sporadic group by \cite[Theorem~2]{09VasGrMaz.t} and
\cite[Theorem~2]{11VasGrSt.t}. If $S$ a group of Lie type over a field of characteristic $p$, then $S\simeq L$ due to \cite[Theorem~3]{11VasGrSt.t} and \cite[Theorem~2]{15VasGr1}.
Then \cite[Corollary~1.1]{15Gr} yields that either $K=1$ or $L=L_4^\tau(q)$.  In the latter case, $K$ must be a $p$-group in view of \cite[Lemma~11]{08Zav.t}. Thus, we may assume that $S$ is
a simple group of Lie type over the field of order $u$ and characteristic $v\neq p$.

Finally, applying Lemma~\ref{l:nilp}, we derive the following assertion.

\begin{lemma} \label{l:a} Let $q$ be odd and let $L$ be one of the simple groups $L_n^\tau(q)$, where $n\geqslant 4$, $S_{2n}(q)$, where $n\geqslant 3$, $O_{2n+1}(q)$, where $n\geqslant 3$, or
$O_{2n}^\tau(q)$, where $n\geqslant 4$.
Suppose that $G$ is a finite group such that $\omega(G)=\omega(L)$, and $K$ and $S$ are as in \eqref{e:ns}. Then either $K$ is nilpotent, or one of the following holds:
\begin{enumerate}
 \item $L=O_{2n}^\tau(q)$, where $n$ is odd, $q\equiv \tau\pmod 8$, and $R_n(\tau q)\cap \pi(S)\subseteq \pi(K)$;
 \item $L=L_n^\tau(q)$, where $1<(n)_2<(q-\tau)_2$, and $R_{n-1}(\tau q)\cap\pi(S)\subseteq \pi(K)$;
 \item $L=L_n^\tau(q)$, where $(n)_2>(q-\tau)_2$ or $(n)_2=(q-\tau)_2=2$, and $R_{n}(\tau q)\cap\pi(S)\subseteq \pi(K)$.
 \end{enumerate}
\end{lemma}

\begin{proof} If $L\neq L_4^\tau(q)$, then this follows from~\cite[Lemma~4.1]{19GrVasZv}.

Let $L=L_4^\tau(q)$.  We show that either for every $r\in\pi(K)$, there is $a\in\omega(S)$ satisfying $\pi(a)\cap\pi(K)=\varnothing$
and $ar\not\in\omega(G)$, in which case $K$ is nilpotent by Lemma~\ref{l:nilp}, or one of~(ii) and (iii) holds.

If $(q-\tau)_2=4$, then the elements of $R_4(\tau q)\cup R_{3}(\tau q)$ are not
adjacent to $2$ in $GK(G)$, so we can take $a=r_4(\tau q)$ if $r$ is coprime to $m_4=(q^2+1)(q+\tau)/(4,q-\tau)$, and $a=r_{3}(\tau q)$ otherwise.

Let $(q-\tau)_2>4$. If $r$ is coprime to $m_4$, then $a=r_4(\tau q)$. If $r$ divides  $m_4$ and there is $s\in (\pi(S)\cap R_{3}(\tau q))\setminus \pi(K)$, then $a=s$.
The case $(q-\tau)_2<4$ is similar with $m_4$, $r_4(\tau q)$ and $R_3(\tau q)$ replacing by  $m_3=(q^3-\tau)/(4,q-\tau)$, $r_3(\tau q)$ and $R_4(\tau q)$ respectively.  Therefore, (ii) or (iii)
holds, and the proof is complete.
\end{proof}

\section{General Case}

Let $G$ and $L$ be as in Theorem~\ref{t:main} and suppose that the solvable radical $K$ of $G$ is not nilpotent. According to the previous section, we may assume that
$L$ is a classical group over a field of odd characteristic $p$ and order $q$ with connected prime graph, $G$ has a normal series
\begin{equation*}
1\leqslant K<H\leqslant G,
\end{equation*}
where $S=H/K$ is a simple group of Lie type over a field of characteristic $v\neq p$ and order $u$, and $G/K$ is a subgroup of $\Aut(S)$. Moreover, we may assume that one
of the assertions (i)--(iii) of the conclusion of Lemma~\ref{l:a} holds. Before we proceed with the proof of Theorem~\ref{t:main}, we introduce some notation which allow us to deal with the cases described in
(i)--(iii) of Lemma~\ref{l:a} simultaneously.

We define $z$ to be the unique positive integer such that $R_{z}(\tau q)\subseteq\pi(L)$ and each $r\in R_{z}(\tau q)$ is not adjacent to $2$ in $GK(L)$.  We also define
$y$ to be  the unique
positive integer such that $R_{y}(\tau q)\subseteq\pi(L)$ and each $r\in R_{y}(\tau q)$ is adjacent to $2$ but not to $p$. Both $z$ and $y$ are well-defined  for
all groups $L$ in (i)--(iii). Indeed, $y=n$,  $z=2n-2$ in (i); $y=n-1$, $z=n$ in (ii); and $y=n$, $z=n-1$ in~(iii). Moreover, the last containments of  all three assertions  in
Lemma~\ref{l:a} can be written uniformly as \begin{equation}\label{eq:ry} R_y(\tau q)\cap \pi(S)\subseteq \pi(K).\end{equation} Observe that $r_z(\tau q)$ is not adjacent to both $2$ and $p$, and
hence $\{p, r_z(\tau q), r_y(\tau q)\}$ is a coclique in $GK(L)$. Also observe that $z, y\geqslant 3$, and so if $r\in R_{z}(\tau q) \cup R_y(\tau q)$, then $r\geqslant 5$.

It is not hard to check that every number in $\omega(L)$ that is a multiple of $r_{z}(\tau q)$ or $r_{y}(\tau q)$
has to divide the only element of $\mu(L)$ which we denote by $m_z$ or $m_y$ respectively. Namely, $$m_z=(q^{n-1}+1)(q+\tau)/4, \quad m_y=(q^n-\tau)/4$$ in (i);
$$\{m_z,m_y\}=\left\{\frac{q^n-1}{(q-\tau)(n,q-\tau)}, \frac{q^{n-1}-\tau}{(n,q-\tau)}\right\}$$ in (ii) and (iii). The numbers $m_z$ and $m_y$ are coprime, so $r_z(\tau q)$ and $r_y(\tau q)$
have no common neighbours in $GK(L)$. Also we note that $m_y$ is even.

The definition of the number $z$ and Lemma \ref{l:str}(ii) imply that $R_z(\tau q)\cap (\pi(K)\cup\pi(G/H))=\varnothing$ and so $k_z(\tau q)\in\omega(S)$. Hence there is a number $k(S)\in\omega(S)$
such that $k_z(\tau q)$ divides $k(S)$ and every $r\in\pi(k(S))$ is not adjacent to 2 in $GK(S)$. Using, for example, \cite[Section 4]{05VasVd.t}, one can see that
this number is uniquely determined and of the form $k_j(u)$ for some positive integer $j$ provided that $S\neq L_2(u)$ and $S$ is not a Suzuki or Ree group.
If $k(S)\neq k_j(u)$, then $k(S)=v$ when $S=L_2(u)$, and $k(S)=f_+(u)$ or $f_-(u)$, where $f_{\pm}(u)$ is as follows,  when $S$ is a Suzuki or Ree group.

\renewcommand{\arraystretch}{1.5}
$$\begin{array}{|l@{\hspace{10mm}}l|}
\hline
S&f_{\pm}(u)\\
\hline
^2B_2(u)& u\pm\sqrt{2u}+1\\
^2G_2(u)&u\pm\sqrt{3u}+1\\
^2F_2(u)& u^2\pm\sqrt{2u^3}+u\pm \sqrt{2u}+1\\
\hline
\end{array}$$

\begin{lemma} \label{l:field} $R_{y}(\tau q)\subseteq\pi(K)$.
\end{lemma}

\begin{proof} Suppose that $r\in R_{y}(\tau q)\setminus\pi(K)$. Then \eqref{eq:ry} yields $r\in\pi(G/H)\setminus\pi(H)$. Since $S$ is a simple group of Lie type, $r\not\in\pi(S)$ and $r\geqslant 5$,
there is a field automorphism $\varphi$ of $S$ of order $r$ lying in $G/K$ and $u=u_0^r$. The centralizer $C$ of $\varphi$ in $S$ is the group of the same Lie type as $S$ over the subfield of the
base field of order $u_0$. Since $r\cdot\pi(C)\subseteq\omega(G)$, it follows that every $s\in\omega(C)$ must divide $m_y$.

Recall that $k_z(\tau q)$ divides $k(S)$. It is clear that $k(S)\neq v$ since $v\in\pi(C)$. If $S$ is a Suzuki or Ree group and $k(S)=f_+(u)$ or $f_-(u)$, then either $f_+(u_0)$ or $f_-(u_0)$ divides
$k(S)$ and so divides $m_z$. On the other hand, both $f_+(u_0)$ and $f_-(u_0)$ lie in $\omega(C)$.

Suppose that $k(S)=k_j(u)$ and take $t\in R_j(u_0)$. It is clear that $t\in\pi(C)$. If $(j,r)=1$, then Lemma~\ref{l:pras}
yields $R_j(u_0)\subseteq R_j(u)$, so $t$ divides $k_j(u)$, and hence it divides $m_z$. This is a contradiction because $(m_y,m_z)=1$. Assume that $r$ divides $j$.
If $S$ is an exceptional group of Lie type, then using  \cite[Table~7]{05VasVd.t} and the
condition $r\geqslant 5$, we conclude that $r=5$ and $S=E_8(u)$, or $r=7$ and $S=E_7(u)$. In either case, $r\in\pi(S)$, a contradiction.
If $S$ is one of the groups $L_m^\varepsilon(u)$, $S_{2m}(u)$, $O_{2m+1}(u)$, $O_{2m}^\varepsilon(u)$, then $(j)_{2'}\leqslant m$,
and so $r\leqslant m$. Since $r$ divides $u(u^{r-1}-1)$, it follows that $r\in\pi(S)$, and the proof is complete.
\end{proof}

\begin{lemma}\label{l:nil1} Let $F$ be the Fitting subgroup of $K$.
\begin{enumerate}
\item  $R_{y}(\tau q)\subseteq\pi(F)\setminus\pi(H/F)$.
\item If $s\in\pi(G)$ and $(s, m_y)=1$, then Sylow $s$-subgroups of $G$ are cyclic.
\end{enumerate}
 \end{lemma}

\begin{proof} Suppose that $K$ is not nilpotent. We will use the notation from the proof of Lemma~\ref{l:nilp}, which means the following. The Fitting subgroup $F$ of $K$ is a proper subgroup of
$K$, and $\widetilde G=G/F$, $\widetilde K=K/F$. We choose $\widetilde T$ to be a minimal normal subgroup of $\widetilde G$ contained in $\widetilde K$ and let $T$ be its preimage in $G$. It is
clear that $\widetilde T$ is an elementary abelian $t$-group for some prime $t$. Given $r\in \pi(F)\setminus \{t\}$, denote the Sylow $r$-subgroup of $F$ by $R$, its centralizer in $G$ by $C_r$ and
the image of $C_r$ in $\widetilde G$ by $\widetilde C_r$. Since $F$ is a proper subgroup of $F$, there is $r\in\pi(F)\setminus\{t\}$ such that $\widetilde C_r \cap \widetilde T = 1$, and we fix some
$r$ enjoying this condition.

Suppose first that $C_{\widetilde G}(\widetilde T)\leqslant \widetilde K$. Then $r$ divides $m_z$, otherwise due to the standard arguments from the proof of Lemma~\ref{l:nilp} we have
$rr_{z}(\tau q)\in\omega(G)$, which is not the case. By Lemma~\ref{l:field}, we have that $R_{y}(\tau q)\subseteq\pi(K)$. If $w\in R_{y}(\tau q)$ and $W$ is a Sylow $w$-subgroup of $K$, then
$R\rtimes
W$ is a Frobenius group because $(m_y,m_z)=1$. It follows that $W$ is cyclic. Thus, $N_G(W)/C_G(W)$ must be abelian. The Frattini argument implies that $N_G(W)$ contains a nonabelian composition
section $S$, and so does $C_G(W)$. Then $wr_{z}(\tau q)\in\omega(G)$, a contradiction.

Thus, we may assume that $C_{\widetilde G}(\widetilde T)\nleqslant \widetilde K$. Then $C_{\widetilde G}(\widetilde T)$ contains a nonabelian composition factor isomorphic to~$S$, and in particular
$t$ is adjacent to every prime in $\pi(S)$. Since $r_{z}(q)$ divides $|S|$, it follows that $t$ divides $m_z$ and $R_{y}(\tau q)\cap \pi(S)=\varnothing$.

If $R_{y}(\tau q)\cap(\pi(K)\setminus\pi(F))\neq\varnothing$ and $\widetilde W$ is a Sylow $w$-subgroup of $\widetilde K$ for some prime $w$ from this intersection, then the group $\widetilde
T\rtimes \widetilde W$ is Frobenius, and we get a contradiction as above.

Thus,  $R_{y}(\tau q)\subseteq\pi(F)$. Let $w\in R_{y}(\tau q)$ and let $W$ be a Sylow $w$-subgroup of $F$,  while $P$ be a $s$-subgroup of $G$ for some $s\in\pi(G)$ such that $(s,
m_y)=1$. Since $w$ and $s$ are nonadjacent in $GK(G)$, the group $W\rtimes P$ is Frobenius. Therefore, $P$ is cyclic.
\end{proof}

\begin{lemma} \label{l:nil} We may assume that $L=L_6^\tau(q)$ and $z=5$, or $L=L_4^\tau(q)$.
\end{lemma}
\begin{proof}
Suppose that $L\neq L_6(\tau q)$, $L_4(\tau q)$. Then one can easily check using \cite{11VasVd.t} that $t(L)\geqslant4$ and there is a coclique of size 4 in $GK(L)$ that
contains an element of the form $r_y(\tau q)$ and does not contain $p$. Let $\rho$ be such a coclique and $\rho'=\rho\setminus R_y(\tau q)$. Lemma \ref{l:nil1}(i) and Lemma~\ref{l:str} imply that
$\{p\}\cup \rho'\subseteq\pi(S)\setminus(\pi(K)\cup\pi(G/H))$. On the other hand, $p$ is not adjacent at most to one element of $\rho'$, an element of the form $r_{z}(\tau q)$. Hence $p$ is
adjacent to at least two elements of $\rho'$. By Lemma~\ref{l:v}, a Sylow $p$-subgroup of $S$ is not cyclic.
This contradicts Lemma  \ref{l:nil1}(ii).

Suppose that $L=L_6^\tau(q)$ and $z=6$. Since $\{p, r_6(\tau q), r_5(\tau q)\}$ and $\{r_3(\tau q), r_4(\tau q), r_5(\tau q)\}$ are cocliques in $GK(L)$, it follows that $\{p\}\cup R_3(\tau q)\cup
R_4(\tau q)\subseteq \pi(S)\setminus(\pi(K)\cup\pi(G/H))$. However, $p$ is adjacent to $r_3(\tau q)$ and $r_4(\tau q)$ in $GK(L)$, and we derive a contradiction as above.
\end{proof}

\section{Small dimensions}

In this section we handle the remaining case $L=L_6^\tau(q)$ and $z=5$, or $L=L_4^\tau(q)$. If $L$ is one of the groups $L_4(3)$, $U_4(3)$, $L_4(5)$, and $U_6(5)$, then $L$ has disconnected
prime graph, so $K$ is nilpotent. If $L=U_4(5)$ or $L_6(3)$, then $K=1$ (see \cite{05Vas1.t} and \cite{06DFS} respectively). If $L=L_6(5)$ or $L=U_6(3)$, then $z=6$. Thus we may assume that $q>5$,
and in particular $\gexp_p(L)\leqslant q$.

We begin with some more notation and two auxiliary lemmas. Choose $x\in\{1,2\}$ such that $2\not\in R_x(\tau q)$.
By definition and Lemma \ref{l:bz}, $R_x(\tau q)$ is not empty and consists of odd primes. Furthermore, it is not hard to see that
$x=2$ if $z=n$ and $x=1$ if $z=n-1$. If $x=2$, then it is clear that $r_x(\tau q)$ is adjacent to $r_z(\tau q)$ but not to $r_y(\tau q)$
in $GK(L)$. If $x=1$, then $r\in R_x(\tau q)$ is adjacent to $r_z(\tau q)$ if and only if $(q-\tau)_r>(n)_r$ and not adjacent to  $r_y(\tau q)$ if and only if $(q-\tau)_r\geqslant (n)_r$
(see, for example, \cite[Propositions 4.1 and 4.2]{05VasVd.t}).
Since $n=6$ or $4$, we conclude that $r_x(\tau q)$ is always adjacent to $r_z(\tau q)$ but not to $r_y(\tau q)$ unless $n=6$, $r_x(\tau q)=3$ and $(q-\tau)_3=3$ in which case $r_x(\tau q)$ is
not adjacent to both $r_z(\tau q)$ and $r_y(\tau q)$.

\begin{lemma}\label{l:L6_except}
Let $L=L_6^\tau(q)$, $z=5$, $\sigma=\pi(L)\setminus R_{6}(\tau q)$ and  suppose that $S\neq{}^3D_4(u)$ is an exceptional group of Lie type such that $k_{5}(\tau
q)\in\omega(S)\subseteq\omega (L)$. Then the following holds:
\begin{enumerate}
 \item If $S\neq G_2(u)$, ${}^2B_2(u)$, then $\gexp_{\sigma}(L)<\gexp(S)$.
 \item If $S= G_2(u)$, ${}^2B_2(u)$, $\eta=\sigma\setminus R_1(\tau q)$ and $(q-\tau,5)=1$, then $\gexp_{\eta}(L)<\gexp(S)$.
 \end{enumerate}
\end{lemma}

\begin{proof}
Write $k=k_{5}(\tau q)$ and recall that $k(S)$ is the number in $\omega(S)$ such that $k$ divides $k(S)$ and every $r\in\pi(k(S))$ is not adjacent to $2$ in $GK(S)$.
By Lemma \ref{l:k_n}, we have $$k=\frac{q^5-\tau}{(q-\tau)(5,q-\tau)}.$$

(i) It is easy to see using Lemmas \ref{l:spec_l6} and \ref{l:r-part}  that $$\gexp_{\sigma}(L)=\gexp_p(L)\cdot (3,q+\tau)(q^5-\tau)(q^2+\tau q+1)(q^2+1)(q+\tau).$$

If $\tau=+$, then $k>q^4/5$ and $$\gexp_{\sigma}(L)< q\cdot 3\cdot q^5\cdot \frac{q^3-1}{q-1}\cdot \frac{q^4-1}{q-1}<3q^{11} \left(\frac{q}{q-1}\right)^2\leqslant 3q^{11}\cdot
\left(\frac{7}{6}\right)^2=\frac{49q^{11}}{12}.$$
So $q^4<5k(S)$ and $$\gexp_{\sigma}(L)<\frac{49\cdot \left(5k(S)\right)^{11/4}}{12}<342\left(k(S)\right)^{11/4}.$$

Similarly, if $\tau=-$, then $k>7q^4/40$ and $$\gexp_{\sigma}(L)<q\cdot 3\cdot (q^5+1)(q^2-q+1)(q^2+1)(q-1)<3q^{11}.$$ It follows that $\gexp_{\sigma}(L)<3\cdot
\left(40k(S)/7\right)^{11/4}<363\left(k(S)\right)^{11/4}.$

In either case, we have $\gexp_{\sigma}(L)<\gexp(S)$ unless \begin{equation}\label{e:sexp}\gexp(S)<363\left(k(S)\right)^{11/4}.\end{equation}
Observe that $2^{11/4}<7$.

Let $S=E_8(u)$. Then $\gexp(S)>2u^{80}$ by Lemma \ref{l:exp_e} and $k(S)<2u^8$. It follows from \eqref{e:sexp} that $$2u^{80}<363(2u^8)^{11/4},$$ and so $u^{58}<182\cdot 7$. Similarly,
if $S=E_7(u)$, then $\gexp(S)>3u^{48}$ by Lemma \ref{l:exp_e} and $k(S)<2u^6$, and we derive that $$3u^{48}<363(2u^6)^{11/4},$$ or equivalently, $u^{63/2}<121\cdot 7$. In both cases, we have a
contradiction.

Let $S=E^\varepsilon_6(u)$. Then $\gexp(S)>u^{26}$ and $k(S)\leqslant 2u^6$, and so $$u^{26}<363(2u^6)^{11/4},$$ or equivalently, $u^{19}<363^2\cdot 2^{11/2}$. The last inequality  yields
$u=2$. If $S=F_4(u)$, then $\gexp(S)>3u^{16}$ and $k(S)\leqslant u^4+1$. Thus $$3u^{16}<363(u^4+1)^{11/4},$$ and so again $u=2$. In either case, we have $7q^4/40<k<k(S)<128$, contrary to the
fact that $q>5$.

If $S={}^2F_4(u)$, then $\gexp(S)>5u^{10}$ and $k(S)\leqslant 2u^2$. It follows that $5u^{10}<363(2u^2)^{11/4}$, whence $u^{9/2}<73\cdot 7$. This is a contradiction since $u\geqslant 8$.

If $S={}^2G_2(u)$, then $\gexp(S)>2u^{4}$ and $$k(S)\leqslant u+\sqrt{3u}+1< 2u.$$  We have $2u^{4}<363(2u)^{11/4}$, whence $u^5<(363/2)^4\cdot 2^{11}$, and so $u=3^3$
or $u=3^5$. In fact, using $u+\sqrt{3u}+1$ instead of $2u$, one can check that \eqref{e:sexp} does not hold for $u=3^5$. It follows that $k<2\cdot 3^3$, but we saw above that this is false.

(ii) Since  $z=5$, if follows that $(q-\tau)_2=2$, and so the $R_1(\tau q)$-part of $\gexp_\sigma(L)$ is equal to $(q-\tau)(5,q-\tau)(3,q-\tau)/2$. Hence
$$\gexp_{\eta}(L)\leqslant \frac{2\gexp_{\sigma}(L)}{q-\tau}=\frac{2\gexp_p(L)\cdot (3,q+\tau)(q^5-\tau)(q^2+\tau q+1)(q^2+1)(q+\tau)}{q-\tau}.$$
If $\tau=+$, then $k>q^4$ and $$\gexp_{\eta}(L)< 6q\cdot\frac{q^5-1}{q-1}\cdot \frac{q^3-1}{q-1}\cdot \frac{q^4-1}{q-1}\leqslant 6q^{10}\cdot
\left(\frac{7}{6}\right)^3<10q^{10}<10k^{5/2}.$$

If $\tau=-$, then $k>7q^4/8$ and $$\gexp_{\eta}(L)< 6q\frac{q^5+1}{q+1}\cdot (q^2-q+1)(q^2+1)(q-1)\leqslant 6q^{10}<6\left(\frac{8k}{7}\right)^{5/2}<9k^{5/2}.$$

In either case, $\gexp_{\eta}(L)<\gexp(S)$ unless \begin{equation}\label{e:eexp}\gexp(S)<10\left(k(S)\right)^{5/2}.\end{equation}

If $S=G_2(u)$, then $\gexp(S)\geqslant 7(u^{6}-1)/(3,u^2-1)$ and $k(S)\leqslant u^2+u+1$. So \eqref{e:eexp} yields $7(u^6-1)<30(u^2+u+1)^{5/2}$, or equivalently,
$$7(u^2-1)(u^2-u+1)<30(u^2+u+1)^{3/2}.$$ The last inequality is not valid if $u\geqslant 16$, and hence $u\leqslant 13$, which implies that $k\leqslant 183$. This is a contradiction
since $k\geqslant 7q^4/8$.

If $S={}^2B_2(u)$, then $\gexp(S)=4(u^2+1)(u-1)$ and $k(S)\leqslant u+\sqrt{2u}+1$. So we have $4(u^2+1)(u-1)<30(u+\sqrt{2u}+1)^{5/2}$, or equivalently,
$$2(u-\sqrt{2u}+1)(u-1)<15(u+\sqrt{2u}+1)^{3/2}.$$ It follows that $u=2^3$ or $u=2^5$, and therefore, $k\leqslant 41$, a contradiction.
\end{proof}

\begin{lemma}\label{l:L4_except}
Let $L=L_4^\tau(q)$, $\sigma=\pi(L)\setminus R_{y}(\tau q)$ and suppose that $S\neq{}^3D_4(u)$ is an exceptional group of Lie type such that $k_{z}(\tau
q)\in\omega(S)\subseteq\omega (L)$. Then the following holds:
\begin{enumerate}
 \item if $S\neq G_2(u)$, ${}^2B_2(u)$, then $\gexp_{\sigma}(L)<\gexp(S)$;
 \item if $S= G_2(u)$, ${}^2B_2(u)$, $\eta=\sigma\setminus R_x(\tau q)$, then $\gexp_{\eta}(L)<\gexp(S)$.
 \end{enumerate}
\end{lemma}

\begin{proof}  Write $k=k_{z}(\tau q)$.

(i) If $z=4$, then $k=(q^2+1)/2>q^2/2$ and $$\gexp_{\sigma}(L)=\gexp_p(L)\cdot (3,q-\tau)(q^4-1)/2<3q^5/2<3(2k)^{5/2}/2<9k^{5/2}.$$

Let $z=3$. If $\tau =+$, then $k=(q^2+q+1)/(3,q-1)>q^2/3$ and $$\gexp_{\sigma}(L)=\gexp_p(L)\cdot (q^2+q+1)(q^2-1)<\frac{7q^5}{6}<\frac{7(3k)^{5/2}}{6}<19k^{5/2}.$$
If $\tau=-$, then $k=(q^2-q+1)/(3,q+1)>7q^2/24$ and $$\gexp_{\sigma}(L)=\gexp_p(L)\cdot (q^2-q+1)(q^2-1)<q^5<\left(\frac{24k}{7}\right)^{5/2}<22k^{5/2}.$$

Thus  $\gexp_{\sigma}(L)<\gexp(S)$ unless \begin{equation}\label{e:sexp4}\gexp(S)<22\left(k(S)\right)^{5/2}.\end{equation}
Since \eqref{e:sexp4} is stronger than \eqref{e:sexp}, we may assume  that $S$ is one of the groups $E_6^\varepsilon(2)$, $F_4(2)$, $^2G_2(3^3)$.
It is easy to check that  \eqref{e:sexp4} does not hold for all these groups.

(ii) If $z=4$, then $x=2$ and the $R_{x}(\tau q)$-part of $\gexp(L)$ is equal to $(q+\tau)/2$. So $$\gexp_{\eta}(L)=\gexp_p(L)\cdot
(3,q-\tau)(q^2+1)(q-\tau)<\frac{7q^4}{2}<\frac{7(2k)^2}{2}=14k^2.$$

Similarly, if $z=3$ and $\tau =+$, then $$\gexp_{\eta}(L)=2\gexp_p(L)\cdot
(q^2+q+1)(q+1)<2q^4\left(\frac{7}{6}\right)^2<2\cdot (3k)^2\cdot \left(\frac{7}{6}\right)^2<25k^2.$$
And if $z=3$ and $\tau =-$, then $$\gexp_{\eta}(L)=2\gexp_p(L)\cdot
(q^2-q+1)(q-1)<2q^4<2\cdot (24k/7)^2<24k^2.$$

Thus $\gexp_{\eta}(L)<\gexp(S)$ unless \begin{equation}\label{e:eexp4}\gexp(S)<25\left(k(S)\right)^{2}.\end{equation}

If $S=G_2(u)$, then arguing  as in the proof of Lemma \ref{l:L6_except}, we derive that \eqref{e:eexp4} yields $$7(u^2-1)(u^2-u+1)<75(u^2+u+1),$$
whence $u\leqslant 4$. But then $k\leqslant u^2+u+1\leqslant 21$. Since $k\geqslant 19$, we see that $u=4$ and $k$ divides $21$. This contradicts to the fact that $k$ is coprime to $3$.

Similarly, if $S={}^2B_2(u)$, then
$$ 4(u-\sqrt{2u}+1)(u-1)<25(u+\sqrt{2u}+1).$$ It follows that $u\leqslant 8$, and so $k\leqslant 13$, which is a contradiction.
\end{proof}

The next step in our proof is the following lemma. In fact, this lemma remains valid without assumption that $L=L_6^\tau(q)$, $L_4^\tau(q)$ but we need it only
in this section.

\begin{lemma}\label{l:schur}
If $r\in\pi(K)$,  $(r,m_y)=1$ and $r$ does not divide the order of the Schur multiplier $M(S)$ of $S$, then $r\not\in\pi(S)$ and $\gexp_r(K)\cdot\omega(S)\subseteq\pi(G)$.
\end{lemma}

\begin{proof}
By Lemma \ref{l:nil1}(ii), Sylow $r$-subgroups of $G$ are cyclic. In particular, if
$R$ is a Sylow $r$-subgroup of $K$, then $R$ is cyclic. By Frattini argument, we derive that $C=C_G(R)$ has a composition factor
isomorphic to $S$. A Hall $r'$-subgroup $A$ of the solvable radical of $C$ centralizes the Sylow $r$-subgroup of this radical, so $A$ is normal in $C$.
Factoring out $C$ by $A$, we get a central extension of $S$ by an $r$-subgroup. Denote this central extension by
$\widetilde C$. The derived series of $\widetilde C$ terminates in a perfect central extension of $S$ by an $r$-group. By the hypothesis, this perfect central
extension is isomorphic to $S$, and therefore $\widetilde C$ includes subgroup isomorphic to $R\times S$. Now the lemma follows.
\end{proof}

Now we are ready to complete the proof of Theorem~\ref{t:main}. Applying Lemmas \ref{l:str}, \ref{l:field} and \ref{l:nil1}, we conclude that $p$ and $r_z(\tau q)$ lie in $\pi(S)\setminus(\pi(K)\cup\pi(G/H))$ and the corresponding Sylow subgroups of $S$ must be cyclic.

Suppose that $S$ is a classical group. Let $w\in R_y(\tau q)$. Then $w\in \pi(F)\setminus \pi(H/F)$ and writing $\widetilde G=G/O_{w'}(K)$, $\widetilde K=K/O_{w'}(K)$,  $\widetilde
W=O_{w}(\widetilde K)$, we have that $N=\widetilde K/\widetilde W$ is a $w'$-subgroup. By the Hall-Higman Lemma 1.2.3 \cite{56HalHig}, it follows that $C_{\widetilde K}(\widetilde W)\leqslant \widetilde W$. Let $r\in\pi(N)\cap\pi(m_y)$ and $R$ be a Sylow $r$-subgroup of $N$. By the Frattini argument, there is an element $g\in N_{\widetilde G/\widetilde W}(R)$ of order $r_z(\tau q)$. Since $rr_z(\tau q)\not\in\omega(G)$, $R\rtimes\langle g\rangle $ is a Frobenius group. Applying Lemma \ref{l:frob} yields $wr_z(\tau q)\in\omega(G)$, which is not the case. Thus $\pi(N)\cap\pi(m_y)=\varnothing$. In particular, $2\not\in\pi(N)$. Furthermore, by Lemma \ref{l:schur}, it follows that $\pi(N)\cap\pi(S)\subseteq\pi(M(S))$.

Assume that $r_z(\tau q)\neq v$. At least one of the numbers $r_z(\tau q)$ and $p$, denote this number by $r$, divides the order of a proper parabolic subgroup of $S$ (cf. \cite[Lemma 3.8]{15Vas}). By Lemma \ref{l:hh1}, we see that $rw\in\omega(G)$ unless $v\in\pi(N)$. In this case, by the results of the preceding paragraph, $v$ is odd and $v$ divides $|M(S)|$. It follows that $v=3$ and $S$ is one of the groups $L_2(9)$, $U_4(3)$, and $S_6(3)$. But then $v$-subgroup of $S$ is not cyclic, contrary to the fact that $v$ is coprime to $m_y$ and  Lemma \ref{l:nil1}. If $r_z(\tau q)=v$, then $S=L_2(v)$ since $r_z(\tau q)$ is not adjacent to $2$ and the corresponding Sylow subgroup is cyclic. Furthermore, in this case $\pi(N)\cap\pi(S)=1$. Applying Lemma \ref{l:l_2v}, we see that  $\{p,r_z(\tau q), w\}$ is not a coclique in $GK(G)$, a contradiction.

If $S={}^3D_4(u)$, then $r_z(\tau q)$ divides $u^{4}-u^2+1$, and so $p$ divides $u^6-1$. This implies that Sylow $p$-subgroups of $S$ are not cyclic (see the structure of maximal tori of $^3D_4(q)$ in
\cite{87DerMich}). Thus we may assume that $S$ is an exceptional group of Lie type other than $^3D_4(u)$.

Let $L=L_6^\tau(q)$ and $z=5$.  Since $R_6(\tau q)\cap\pi(S)=\varnothing$, it follows from
Lemma~\ref{l:L6_except}(i)  that $S$ is $^2B_2(u)$ or $G_2(u)$. We claim that $R_1(\tau q)\cap\pi(S)=\varnothing$. Recall that $z=5$ yields $2\not\in R_1(\tau q)$.

Assume that $r\in R_1(\tau q)\cap\pi(S)$. Since $r$ does not divide $m_6(\tau q)$, the corresponding Sylow
subgroup of $G$ is cyclic. This, in particular, implies that $r\neq v,3$.

If $r\in\pi(K)$, then by Lemma \ref{l:schur}, we have $r\in\pi(M(S))$. This is a contradiction, since the Schur multiplier of $^2B_2(u)$ or $G_2(u)$ is either trivial, or a
$2$-group, or a $3$-group.

Suppose that $r\in \pi(G/H)$. Since $r$ is odd, it follows that $u=u_0^r$ for some $u_0$ and $G/K$ contains an automorphism $\varphi$ of $S$ of order $r$ such that $C_S(\varphi)$ is
$^2B_2(u_0)$ or $G_2(u_0)$ respectively. It is not hard to check using Lemmas \ref{l:spec_2b2} and \ref{l:spec_g2}, that there is $t\in\pi(k(S))\setminus\pi(C_S(\varphi))$ (for example, if
$S=G_2(u)$ and $k(S)=u^2+\varepsilon u+1$, then we can take $t=r_{3r}(\varepsilon u_0)$). Observe that $t\in  R_1(\tau q)\cup R_5(\tau q)$ since $k(S)$
divides $m_5(\tau q)$. Let $T$ be a Sylow $t$-subgroup of $S$.
By the Frattini argument, there is an element $g\in N_{G/K}(T)$ of order $r$. The choice of $t$ implies  that $T\rtimes\langle g\rangle$ is a Frobenius group.  By the result of the previous
paragraph, both $r$ and $t$ are coprime to $|K|$, so we may assume that this Frobenius group acts on the Sylow $w$-subgroup of $K$ for some $w\in R_6(\tau q)$. Applying Lemma
\ref{l:frob},
we see that either $tw$ or $rw$ lies in $\pi(G)$, a contradiction.

Thus if $r\in R_1(\tau q)\cap \pi(S)$, then $r\not\in\pi(K)\cup \pi(G/H)$ and $r\neq 3$. It follows that $r$ is adjacent to both $r_z(\tau q)$ and $p$ in $GK(S)$, while $r_z(\tau q)$
and $p$ are not adjacent. This situation is impossible in the graph $GK({}^2B_2(u))$ since its connected components are cliques (see Lemma \ref{l:spec_2b2}). If $S=G_2(u)$, then by Lemma
\ref{l:spec_g2}, the numbers $p$ and  $r_z(\tau
q)$ divides $u^2-u+1$ and $u^2+u+1$ respectively, or vice versa, and so they do not have common neighbours. Thus we proved that $R_1(\tau q)\cap\pi(S)=\varnothing$.

To apply Lemma~\ref{l:L6_except}(ii) and derive a final contradiction for $L_6^\tau(q)$, it remains to show that $(5,q-\tau)=1$. Suppose that $5\in R_1(\tau q)$. Then $5\not\in\pi(S)$, and therefore
$S=G_2(u)$. Furthermore, since $k(S)$ divides $m_5(\tau q)$ and the ratio $m_5(\tau q)/k_5(\tau q)=q-\tau$ is coprime to $|S|$, it follows that $k_5(\tau q)=k(S)=u^2+\varepsilon u+1$ for some
$\varepsilon\in\{+,-\}$. Also by Lemmas \ref{l:r-part} and \ref{l:spec_g2}, we have that $\gexp_5(L)=(m_5(\tau q))_5=5(q-\tau)_5$, and so $k_5(\tau q)\cdot\gexp_5(L)\in\omega(G)$ but
$p\cdot\gexp_5(L)\not\in\omega(L)$. Assume that $\gexp_5(K)<\gexp_5(G)$. Then $5\in\pi(G/H)$, $G/K$ contains a field automorphism $\varphi$ of $S$ of order $5$ and $u^2+\varepsilon
u+1\in\omega(C_S(\varphi))$. As we remarked previously, this is not the case. Thus $\gexp_5(K)=\gexp_5(G)$. Applying Lemma \ref{l:schur} yields $p\cdot\gexp_5(L)\in\omega(G)\setminus\omega(L)$. This
completes the proof for $L_6^\tau(q)$.

The proof for $L_4^\tau(q)$ follows exactly the same lines with $R_x(\tau q)$ in place of $R_1(\tau q)$ and Lemma~\ref{l:L4_except} in place of Lemma~\ref{l:L6_except}.


\begin{thebibliography}{10}

\bibitem{87Asch}
M.~Aschbacher, Chevalley groups of type {$G_2$} as the group of a trilinear
  form, \emph{J. Algebra} \textbf{109} (1987), no.~1, 193--259.

\bibitem{86Bang}
A.~S. Bang, {Taltheoretiske Unders{\o}gelser}, \emph{Tidsskrift Math.}
  \textbf{4} (1886), 70--80, 130--137.

\bibitem{81Bra}
R.~Brandl, Finite groups all of whose elements are of prime power order,
  \emph{{Boll. Unione Mat. Ital., V. Ser., A}} \textbf{18} (1981), 491--493.

\bibitem{08But.t}
A.~A. Buturlakin, {Spectra of finite linear and unitary groups}, \emph{Algebra
  Logic} \textbf{47} (2008), no.~2, 91--99.

\bibitem{10But.t}
A.~A. Buturlakin, {Spectra of finite symplectic and orthogonal groups},
  \emph{Siberian Adv. Math.} \textbf{21} (2011), no.~3, 176--210.

\bibitem{06DFS}
M.~R. {Darafsheh}, Y.~{Farjami}, and A.~{Sadrudini}, {On groups with the same
  set of order elements}, \emph{{Int. Math. Forum}} \textbf{1} (2006),
  no.~25-28, 1325--1334.

\bibitem{84Der}
D.~I. Deriziotis, {Conjugacy classes of centralizers of semisimple elements in
  finite groups of {L}ie type}, {Vorlesungen Fachbereich Math. Univ. Essen},
  vol.~11, {Universit{\"a}t Essen Fachbereich Mathematik}, Essen, 1984.

\bibitem{87DerMich}
D.~I. Deriziotis and G.~O. Michler, {Character table and blocks of finite
  simple triality groups $^3D_4(q)$}, \emph{Trans. Amer. Math. Soc.}
  \textbf{303} (1987), 39--70.

\bibitem{83GorLy}
D.~Gorenstein and R.~Lyons, {Local structure of finite groups of characteristic
  {$2$} type}, vol.~42, {Memoirs of the American Mathematical Society}, no.
  276, American Mathematical Society, Providence, RI, 1983.

\bibitem{13Gor.t}
I.~B. Gorshkov, {Recognizability of alternating groups by spectrum},
  \emph{Algebra Logic} \textbf{52} (2013), no.~1, 41--45.

\bibitem{15Gr}
M.~A. Grechkoseeva, {On element orders in covers of finite simple groups of Lie
  type}, \emph{J. Algebra Appl.} \textbf{14} (2015), 1550056 [16 pages].

\bibitem{16Gr.t}
M.~A. Grechkoseeva, {On spectra of almost simple groups with symplectic or
  orthogonal socle}, \emph{Siberian Math. J.} \textbf{57} (2016), no.~4,
  582--588.

\bibitem{12GrLyt.t}
M.~A. Grechkoseeva and D.~V. Lytkin, {Almost recognizability by spectrum of
  finite simple linear groups of prime dimension}, \emph{Siberian Math. J.}
  \textbf{53} (2012), no.~4, 645--655.

\bibitem{15VasGr1}
M.~A. Grechkoseeva and A.~V. Vasil'ev, {On the structure of finite groups
  isospectral to finite simple groups}, \emph{J. Group Theory} \textbf{18}
  (2015), no.~5, 741--759.

\bibitem{19GrVasZv}
M.~A. Grechkoseeva, A.~V. Vasil'ev, and M.~A. Zvezdina, {Recognition of
  symplectic and orthogonal groups of small dimensions by spectrum}, \emph{J.
  Algebra Appl.} \textbf{18} (2019), no.~12, 1950230 [33 pages].

\bibitem{56HalHig}
P.~Hall and G.~Higman, {On the $p$-length of $p$-soluble groups and reduction
  theorem for Burnside's problem}, \emph{Proc. London Math. Soc.} \textbf{6}
  (1956), no.~3, 1--42.

\bibitem{57Hig}
G.~Higman, Finite groups in which every element has prime power order, \emph{J.
  London Math. Soc.} \textbf{32} (1957), 335--342.

\bibitem{82HupBl2}
B.~Huppert and N.~Blackburn, {Finite groups. II}, {Grundlehren der
  Mathematischen Wissenschaften}, vol. 242, Springer-Verlag, Berlin-New York,
  1982.

\bibitem{07Jor}
H.~E. {Jordan}, {Group-characters of various types of linear groups},
  \emph{{Am. J. Math.}} \textbf{29} (1907), 387--405.

\bibitem{90KlLie}
P.~Kleidman and M.~Liebeck, {The subgroup structure of the finite classical
  groups}, {London Mathematical Society Lecture Note Series}, vol. 129,
  Cambridge University Press, Cambridge, 1990.

\bibitem{89Kon.t}
A.~S. Kondrat'ev, {On prime graph components of finite simple groups},
  \emph{Math. USSR-Sb.} \textbf{67} (1990), no.~1, 235--247.

\bibitem{18Lyt}
Yu.~V. Lytkin, {On finite groups isospectral to the simple groups $S_4(q)$},
  \emph{Sib. {\'E}lektron. Mat. Izv.} \textbf{15} (2018), 570--584.

\bibitem{97Maz.t}
V.~D. Mazurov, {Characterizations of finite groups by sets of orders of their
  elements}, \emph{Algebra and Logic} \textbf{36} (1997), no.~1, 23--32.

\bibitem{98Maz.t}
V.~D. Mazurov, {Recognition of finite groups by a set of orders of their
  elements}, \emph{Algebra and Logic} \textbf{37} (1998), no.~6, 371--379.

\bibitem{02Maz.t}
V.~D. Mazurov, {Recognition of finite simple groups $S_4(q)$ by their element
  orders}, \emph{Algebra Logic} \textbf{41} (2002), no.~2, 93--110.

\bibitem{98MazShi}
V.~D. Mazurov and W.~J. Shi, {A note to the characterization of sporadic simple
  groups}, \emph{Algebra Colloq.} \textbf{5} (1998), no.~3, 285--288.

\bibitem{97Roi}
M.~Roitman, {On Zsigmondy primes}, \emph{Proc. Amer. Math. Soc.} \textbf{125}
  (1997), no.~7, 1913--1919.

\bibitem{17Sta}
A.~Staroletov, {On almost recognizability by spectrum of simple classical
  groups}, \emph{Int. J. Group Theory} \textbf{6} (2017), no.~4, 7--33.

\bibitem{10Sta.t}
A.~M. Staroletov, {Groups isospectral to the degree 10 alternating group},
  \emph{Siberian Math. J.} \textbf{51} (2010), no.~3, 507--514.

\bibitem{62Suz}
M.~Suzuki, {On a class of doubly transitive groups}, \emph{Ann. of Math. (2)}
  \textbf{75} (1962), 105--145.

\bibitem{05Vas.t}
A.~V. Vasil'ev, {On connection between the structure of a finite group and the
  properties of its prime graph}, \emph{Siberian Math. J.} \textbf{46} (2005),
  no.~3, 396--404.

\bibitem{05Vas1.t}
A.~V. Vasil'ev, {On recognition of all finite nonabelian simple groups with
  orders having prime divisors at most 13}, \emph{Siberian Math. J.}
  \textbf{46} (2005), no.~2, 246--253.

\bibitem{15Vas}
A.~V. Vasil'ev, {On finite groups isospectral to simple classical groups},
  \emph{J. Algebra} \textbf{423} (2015), 318--374.

\bibitem{09VasGor.t}
A.~V. Vasil'ev and I.~B. Gorshkov, {On recognition of finite simple groups with
  connected prime graph}, \emph{Siberian Math. J.} \textbf{50} (2009), no.~2,
  233--238.

\bibitem{15VasGr.t}
A.~V. Vasil'ev and M.~A. Grechkoseeva, {Recognition by spectrum for simple
  classical groups in characteristic $2$}, \emph{Siberian Math. J.} \textbf{56}
  (2015), no.~6, 1009--1018.

\bibitem{09VasGrMaz.t}
A.~V. Vasil'ev, M.~A. Grechkoseeva, and V.~D. Mazurov, {On finite groups
  isospectral to simple symplectic and orthogonal groups}, \emph{Siberian Math.
  J.} \textbf{50} (2009), no.~6, 965--981.

\bibitem{11VasGrSt.t}
A.~V. Vasil'ev, M.~A. Grechkoseeva, and A.~M. Staroletov, {On finite groups
  isospectral to simple linear and unitary groups}, \emph{Siberian Math. J.}
  \textbf{52} (2011), no.~1, 30--40.

\bibitem{14VasSt.t}
A.~V. Vasil'ev and A.~M. Staroletov, {Almost recognizability of simple
  exceptional groups of Lie type}, \emph{Algebra Logic} \textbf{53} (2015),
  no.~6, 433--449.

\bibitem{11VasVd.t}
A.~V. Vasil'ev and E.~P. Vdovin, {Cocliques of maximal size in the prime graph
  of a finite simple group}, \emph{Algebra Logic} \textbf{50} (2011), no.~4,
  291--322.

\bibitem{05VasVd.t}
A.~V. Vasil'ev and E.~P. Vdovin, {An adjacency criterion for the prime graph of
  a finite simple group}, \emph{Algebra Logic} \textbf{44} (2005), no.~6,
  381--406.

\bibitem{66Wa}
H.~N. Ward, {On Ree’s series of simple groups}, \emph{Trans. Amer. Math. Soc.}
  \textbf{121} (1966), 62--89.

\bibitem{81Wil}
J.~S. Williams, {Prime graph components of finite groups}, \emph{J. Algebra}
  \textbf{69} (1981), 487--513.

\bibitem{08Zav.t}
A.~V. Zavarnitsine, {On recognition by spectrum among covers of finite simple
  unitary and linear groups}, \emph{Dokl. Math.} \textbf{78} (2008), no.~1,
  481--484.

\bibitem{10Zav.t}
A.~V. Zavarnitsine, {A solvable group isospectral to $S_4(3)$}, \emph{Siberian
  Math. J.} \textbf{51} (2010), no.~1, 20--24.

\bibitem{Zs}
K.~Zsigmondy, {Zur Theorie der Potenzreste}, \emph{Monatsh. Math. Phys.}
  \textbf{3} (1892), 265--284.

\bibitem{16Zve.t}
M.~A. Zvezdina, {Spectra of automorphic extensions of finite simple exceptional
  groups of Lie type}, \emph{Algebra Logic} \textbf{55} (2016), no.~5,
  354--366.

\end{thebibliography}

\end{document}